\documentclass[12pt,reqno]{article}

\usepackage[usenames]{color}
\usepackage{amssymb}
\usepackage{graphicx}
\usepackage{amscd}

\usepackage{amsthm}
\newtheorem{theorem}{Theorem}

\newtheorem{proposition}[theorem]{Proposition}
\newtheorem{corollary}[theorem]{Corollary}

\theoremstyle{definition}

\newtheorem{example}[theorem]{Example}

\usepackage[colorlinks=true,
linkcolor=webgreen, filecolor=webbrown,
citecolor=webgreen]{hyperref}

\definecolor{webgreen}{rgb}{0,.5,0}
\definecolor{webbrown}{rgb}{.6,0,0}

\usepackage{color}

\usepackage{float}

\usepackage{graphics,amsmath,amssymb}
\usepackage{amsfonts}
\usepackage{latexsym}
\usepackage[dvips]{epsfig}
\usepackage{epsf}

\newcommand{\seqnum}[1]{\href{http://oeis.org/#1}{\underline{#1}}}

\begin{document}

\begin{center}
\vskip 1cm{\LARGE\bf  Laurent Biorthogonal Polynomials and Riordan Arrays}  \vskip 1cm
\large
Paul Barry\\
School of Science\\
Waterford Institute of Technology\\
Ireland\\

\end{center}
\vskip .2 in

\begin{abstract} We show that Laurent biorthogonal polynomials whose defining three-term recurrence have constant coefficients have coefficient arrays that are Riordan arrays. For each such family of Laurent biorthogonal polynomials we associate in a natural way a family of orthogonal polynomials. We also extend these results to a notion of generalized orthogonal polynomials whose recurrence depends on three parameters.
\end{abstract}

\section{Introduction}

It is well known that orthogonal polynomials on the real line are characterized by their three-term recurrences \cite{Chihara, Gautschi, Szego}. Specifically, if $P_n(x)$ is a family of orthogonal polynomials, with $P_{-1}(x)=0$ and $P_0(x)=1$, then we can find two sequences
$\alpha_n$ and $\beta_n$ such that

$$P_n(x)=(x-\alpha_n) P_{n-1}(x)-\beta_n P_{n-2}(x).$$

\noindent In the case that $\alpha_n=\alpha$ and $\beta_n=\beta$ are constant sequences, then the coefficient array $$\{d_{n,k}\}$$ of the polynomials is a Riordan array \cite{Barry_Moment, Barry_Meixner, SGWW}, where
$$P_n(x)=\sum_{k=0}^n d_{n,k}x^k.$$
This Riordan array will then be of the form
$$\left(\frac{1-\delta x-\epsilon x^2}{1+\alpha x+\beta x^2}, \frac{1}{1+\alpha x+\beta x^2}\right).$$

In this note, we shall look at the situation of Laurent biorthogonal polynomials whose defining recurrences have constant coefficients. Based on the above results for orthogonal polynomials, it is natural to ask if there is a similar connection between Riordan arrays and Laurent biorthogonal polynomials whose defining recurrences have constant coefficients. We shall show that the answer is in the affirmative.  We recall that a family of polynomials $P_n(x)$ is said to be a family of Laurent biorthogonal  polynomials \cite{Hendriksen, Kamioka_14, Kamioka_15, Zhedanov} if there exist two sequences $\alpha_n$ and $\beta_n$ such that
$$P_n(x)=(x-\alpha_n) P_{n-1}(x)- \beta_n x P_{n-2}(x).$$

\noindent The results in this note are typified by the following two propositions.

\begin{proposition} Let $P_n(x)=P_n(x;\alpha, \beta)$ be the family of Laurent biorthogonal  polynomials defined by the recurrence

$$P_n(x)=(x-\alpha) P_{n-1}(x)-\beta x P_{n-2}(x),$$
with $P_0(x)=1$ and $P_1(x)=x-\alpha$. Then the coefficient array of the polynomials $P_n(x)$ is given by the Riordan array
$$\left(\frac{1}{1+\alpha x}, \frac{x(1-\beta x)}{1+\alpha x}\right).$$
\end{proposition}

\noindent Similarly, we have

\begin{proposition} Let $P_n(x)=P_n(x;\alpha, \beta)$ be the family of Laurent biorthogonal  polynomials defined by the recurrence

$$P_n(x)=(x-\alpha) P_{n-1}(x)-\beta x P_{n-2}(x),$$
with $P_0(x)=1$ and $P_1(x)=x-(\alpha+\beta)$. Then the coefficient array of the polynomials $P_n(x)$ is given by the Riordan array
$$\left(\frac{1-\beta x}{1+\alpha x}, \frac{x(1-\beta x)}{1+\alpha x}\right).$$
\end{proposition}

We will concentrate on this latter version, as we find it easier to work with the element $\left(\frac{1-\beta x}{1+\alpha x}, \frac{x(1-\beta x)}{1+\alpha x}\right)$ of the Bell subgroup of the Riordan group. Before we prove the above result, we recall some notation and elements of Riordan group theory. Readers familiar with Riordan groups may skip this section. 

For an integer sequence $a_n$, that is, an element of $\mathbb{Z}^\mathbb{N}$, the power series
$f(x)=\sum_{n=0}^{\infty}a_n x^n$ is called the \emph{ordinary generating function} or g.f. of the sequence.
$a_n$ is thus the coefficient of $x^n$ in this series. We denote this by
$a_n=[x^n]f(x)$. For instance, $F_n=[x^n]\frac{x}{1-x-x^2}$ is the $n$-th Fibonacci number \seqnum{A000045}, while
$S_n=[x^n]\frac{1-\sqrt{1-6x+x^2}}{2x}$ is the $n$-th (large) Schr\"oder number \seqnum{A006318}. 

For a power series
$f(x)=\sum_{n=0}^{\infty}a_n x^n$ with $f(0)=0$ we define the reversion or compositional inverse of $f$ to be the
power series $\bar{f}(x)$ such that $f(\bar{f}(x))=x$. We shall sometimes write this as
$\bar{f}= \text{Rev}f$.
\newline\newline
For a lower triangular matrix $(a_{n,k})_{n,k \ge 0}$ the row sums give the sequence with general term
$\sum_{k=0}^n a_{n,k}$. 
\noindent
The \emph{Riordan group} \cite{SGWW, Spru}, is a set of
infinite lower-triangular integer matrices, where each matrix is
defined by a pair of generating functions
$g(x)=1+g_1x+g_2x^2+\cdots$ and $f(x)=f_1x+f_2x^2+\cdots$ where
$f_1\ne 0$ \cite{Spru}. We assume in addition that $f_1=1$ in what follows. The associated matrix is the matrix whose
$i$-th column is generated by $g(x)f(x)^i$ (the first column being
indexed by 0). The matrix corresponding to the pair $g, f$ is
denoted by $(g, f)$ or $\cal{R}$$(g,f)$. The group law is then given
by
\begin{displaymath} (g, f)\cdot(h, l)=(g, f)(h, l)=(g(h\circ f), l\circ
f).\end{displaymath} The identity for this law is $I=(1,x)$ and the
inverse of $(g, f)$ is $(g, f)^{-1}=(1/(g\circ \bar{f}), \bar{f})$
where $\bar{f}$ is the compositional inverse of $f$.

If $\mathbf{M}$ is the matrix $(g,f)$, and
$\mathbf{a}=(a_0,a_1,\ldots)^T$ is an integer sequence with ordinary
generating function $\cal{A}$ $(x)$, then the sequence
$\mathbf{M}\mathbf{a}$ has ordinary generating function
$g(x)$$\cal{A}$$(f(x))$. The (infinite) matrix $(g,f)$ can thus be considered to act on the ring of
integer sequences $\mathbb{Z}^\mathbb{N}$ by multiplication, where a sequence is regarded as a
(infinite) column vector. We can extend this action to the ring of power series
$\mathbb{Z}[[x]]$ by
$$(g,f):\cal{A}(\mathnormal{x}) \mapsto \mathnormal{(g,f)}\cdot
\cal{A}\mathnormal{(x)=g(x)}\cal{A}\mathnormal{(f(x))}.$$
\begin{example} The so-called \emph{binomial matrix} $\mathbf{B}$ is the element
$(\frac{1}{1-x},\frac{x}{1-x})$ of the Riordan group. It has general
element $\binom{n}{k}$, and hence as an array coincides with Pascal's triangle. More generally, $\mathbf{B}^m$ is the
element $(\frac{1}{1-m x},\frac{x}{1-mx})$ of the Riordan group,
with general term $\binom{n}{k}m^{n-k}$. It is easy to show that the
inverse $\mathbf{B}^{-m}$ of $\mathbf{B}^m$ is given by
$(\frac{1}{1+mx},\frac{x}{1+mx})$.
\end{example}

The proof of the propositions above is dependent on the sequence characterization of Riordan arrays \cite{He, Alter}, and in particular on the A-sequence. One version of the sequence characterization of Riordan arrays is given below.

\begin{proposition} \cite{He} Let $D=[d_{n,k}]$ be an infinite triangular matrix. Then $D$ is a Riordan array if and only if there exist two sequences $A=[a_0,a_1,a_2,\ldots]$ and $Z=[z_0,z_1,z_2,\ldots]$ with $a_0 \neq 0$  such that
\begin{itemize}
\item $d_{n+1,k+1}=\sum_{j=0}^{\infty} a_j d_{n,k+j}, \quad (k,n=0,1,\ldots)$
\item $d_{n+1,0}=\sum_{j=0}^{\infty} z_j d_{n,j}, \quad (n=0,1,\ldots)$.
\end{itemize}
\end{proposition}
The coefficients $a_0,a_1,a_2,\ldots$ and $z_0,z_1,z_2,\ldots$ are called the $A$-sequence and the $Z$-sequence of the Riordan array
$D=(g(x),f(x))$, respectively. Letting $A(x)$ and $Z(x)$ denote the generating functions of these sequences, respectively, we have \cite{Alter} that
$$\frac{f(x)}{x}=A(f(x)), \quad g(x)=\frac{d_{0,0}}{1-xZ(f(x))}.$$ We therefore deduce that
$$A(x)=\frac{x}{\bar{f}(x)},$$ and
$$Z(x)=\frac{1}{\bar{f}(x)}\left[1-\frac{d_{0,0}}{g(\bar{f}(x))}\right].$$
When $(g, f)=(g, xg)$, we then obtain that
$$Z(x)=\frac{A(x)-1}{x}.$$

\noindent We now turn to the proof of the proposition.
\begin{proof}
We let $\left(d_{n,k}\right)$ be the coefficient array of the polynomials $P_n(x)$ that are defined by the recurrence
$$P_n(x)=(x-\alpha) P_{n-1}(x)-\beta x P_{n-2}(x).$$
\noindent
The recurrence implies that
$$\sum_{k=0}^n d_{n,k}x^k=(x-\alpha) \sum_{k=0}^{n-1}d_{n-1,k}x^k-\beta x \sum_{k=0}^{n-2}d_{n-2,k}x^k,$$
or
$$\sum_{k=0}^n d_{n,k}x^k=\sum_{k=0}^{n-1}d_{n-1,k}x^{k+1}-\alpha \sum_{k=0}^{n-1}d_{n-1,k}x^k - \beta \sum_{k=0}^{n-2}d_{n-2,k}x^{k+1},$$
or
$$\sum_{k=0}^n d_{n,k}x^k=\sum_{k=1}^n d_{n-1,k-1}x^k-\alpha \sum_{k=0}^{n-1}d_{n-1,k}x^k - \beta \sum_{k=1}^{n-1}d_{n-2,k-1}x^k.$$
Thus for $0<k<n$, we require that
$$d_{n,k}=d_{n-1,k-1}-\alpha d_{n-1,k}- \beta d_{n-2,k-1}.$$

Now the A-sequence of the Riordan array $\left(\frac{1}{1+\alpha x}, \frac{x(1-\beta x)}{1+\alpha x}\right)$ (and that of $\left(\frac{1-\beta x}{1+\alpha x}, \frac{x(1-\beta x)}{1+\alpha x}\right)$) is generated by the power series
$$A(x)=\frac{x}{\bar{f}(x)},$$ where $f(x)=\frac{x(1-\beta x)}{1+\alpha x}$. Thus
$$A(x)=\frac{2 \beta x}{1-\alpha x - \sqrt{1-(\alpha+2\beta)x+\alpha^2 x^2}},$$ and the A-sequence $[a_0, a_1, a_2,\ldots]$ begins
 $$1, -(\alpha+\beta), \beta(\alpha+\beta), -\beta(\alpha+\beta)(\alpha+2\beta),-\beta (\alpha+\beta)((\alpha^2+5 \alpha \beta+5 \beta^2), \ldots.$$
By the sequence characterization of Riordan arrays, we have
$$d_{n,k+1}=a_0d_{n-1,k}+a_1 d_{n-1,k+1}+a_2 d_{n-1,k+2}+\ldots .$$ We can reverse this relation \cite{Alter} using the coefficients $b_n$ obtained by
$$b_n=-\frac{1}{a_0} \sum_{j=1}^n a_j b_{n-j}, \quad b_0=\frac{1}{a_0},$$ to obtain
$$d_{n-1,k}=a_0d_{n,k+1}-a_1 d_{n,k+1}-(a_2-a_1^2) d_{n,k+2}-(a_3-2a_1a_2+a_1^3)d_{n,k+3}+\ldots$$ which in this case gives
$$d_{n-1,k}=d_{n,k+1}+(\alpha+\beta)d_{n,k+1}+(\alpha+\beta)(\alpha+2\beta)d_{n,k+2}+(\alpha+\beta)(\alpha^2+5\alpha \beta+5\beta^2)d_{n,k+3}+\ldots.$$
Thus we get
\begin{eqnarray*} d_{n+1,k+1}&=&d_{n,k}-(\alpha+\beta)d_{n,k+1}-\beta(\alpha+\beta)d_{n,k+2}-\beta(\alpha+\beta)(\alpha+2\beta)d_{n,k+3}+\ldots \\
&=&d_{n,k}-\alpha d_{n,k+1}-\beta \{d_{n,k+1}+(\alpha+\beta)d_{n,k+2}+\beta(\alpha+\beta)(\alpha+2\beta)d_{n,k+3}+\ldots\},\\
&=&d_{n,k}-\alpha d_{n,k+1}-\beta d_{n-1,k}.\end{eqnarray*}
\end{proof}
\noindent We can derive an expression for $P_n(x)$ using the definition of Riordan arrays. The $(n,k)$-th element $d_{n,k}$ of the coefficient array $\left(\frac{1-\beta x}{1+\alpha x}, \frac{x(1-\beta x)}{1+\alpha x}\right)$ is given by
\begin{eqnarray*}d_{n,k}&=&[x^n] \frac{1-\beta x}{1+\alpha x}\left(\frac{x(1-\beta x)}{1+\alpha x}\right)^k\\
&=& [x^{n-k}]\frac{(1-\beta x)^{k+1}} {(1+\alpha x)^{k+1}}\\
&=& (-1)^{n-k} \sum_{j=0}^{k+1}\binom{k+1}{j}\binom{n-j}{n-k-j}\alpha^{n-k-j}\beta^j. \end{eqnarray*}
Thus we get
$$P_n(x)=\sum_{k=0}^n (-1)^{n-k} \sum_{j=0}^{k+1} \binom{k+1}{j}\binom{n-j}{n-k-j}\alpha^{n-k-j}\beta^j x^k.$$
\noindent We note that an ``$\alpha_{i,j}$" matrix \cite{Alter} for the Riordan array $\left(\frac{1-\beta x}{1+\alpha x}, \frac{x(1-\beta x)}{1+\alpha x}\right)$ is given by
$$\begin{bmatrix} -\beta & 0\\1 &-\alpha \end{bmatrix}.$$
\begin{example} We take the simple case of $\alpha=\beta=1$. Thus we look at the Riordan array
$$\left(\frac{1-x}{1+x}, \frac{x(1-x)}{1+x}\right).$$ \noindent This matrix begins
\begin{displaymath}\left(\begin{array}{ccccccc} 1 & 0 & 0 & 0
&
0 & 0 & \cdots \\-2 & 1 & 0 & 0 & 0 & 0 & \cdots \\ 2 & -4 & 1 &
0 & 0 & 0 &
\cdots \\ -2 & 8 & -6 & 1 & 0 & 0 & \cdots \\ 2 & -12 & 18 & -8 &
1 & 0 & \cdots \\-2 & 16  & -38 & 32 & -10 & 1 &\cdots\\ \vdots
& \vdots &
\vdots & \vdots & \vdots & \vdots &
\ddots\end{array}\right),\end{displaymath} which is the coefficient array of the polynomials
$P_n(x)=P_n(x;1,1)$ which satisfy
$$P_n(x)=(x-1)P_{n-1}(x)-xP_{n-2}.$$
\noindent The inverse array begins
\begin{displaymath}\left(\begin{array}{ccccccc} 1 & 0 & 0 & 0
&
0 & 0 & \cdots \\2 & 1 & 0 & 0 & 0 & 0 & \cdots \\ 6 & 4 & 1 &
0 & 0 & 0 &
\cdots \\ 22 & 16 & 6 & 1 & 0 & 0 & \cdots \\ 90 & 68 & 30 & 8 &
1 & 0 & \cdots \\394 & 304  & 146 & 48 & 10 & 1 &\cdots\\ \vdots
& \vdots &
\vdots & \vdots & \vdots & \vdots &
\ddots\end{array}\right),\end{displaymath} where we see that the moments are given by
the large Schr\"oder numbers $1,2,6,22,90,\ldots$, \seqnum{A06318}.
\end{example}
\noindent Other examples of arrays of this type may be found for instance in the On-Line Encyclopedia of Integer Sequences \cite{SL1, SL2}, where the last two number triangles are \seqnum{A080246} and \seqnum{A080247}, respectively.
\section{Associated orthogonal polynomials}
We can associate a family of orthogonal polynomials to the family $P_n(x;\alpha, \beta)$ by taking a suitable  $\beta$-fold inverse binomial transform.
\begin{proposition} The matrix
$$\left(1, \frac{x}{1-\beta x}\right)^{-1} \cdot \left(\frac{1-\beta x}{1+\alpha x}, \frac{x(1-\beta x)}{1+\alpha x}\right)$$ is the coefficient array of the family of orthogonal polynomials $\tilde{P}_n(x)=\tilde{P}_n(x;\alpha, \beta)$ that satisfies the three-term recurrence
$$\tilde{P}_n(x)=(x-(\alpha+2\beta))\tilde{P}_{n-1}(x)-\beta(\alpha+\beta) \tilde{P}_{n-1}(x),$$
with $\tilde{P}_0(x)=1$ and $\tilde{P}_1(x)=x-(\alpha+\beta)$.
\end{proposition}
\begin{proof}
We have
$$\left(1,\frac{x}{1-\beta x}\right)^{-1}\cdot \left(\frac{1-\beta x}{1+\alpha x}, \frac{x(1-\beta x)}{1+\alpha x}\right)=
\left(\frac{1}{1+(\alpha+\beta)x},\frac{x}{(1+\beta x)(1+(\alpha+\beta )x}\right).$$
The inverse of this Riordan array has tri-diagonal production matrix which begins
\begin{displaymath}\left(\begin{array}{ccccccc} \alpha+\beta & 1 &
0
& 0 & 0 & 0 & \ldots \\ \beta(\alpha+\beta)  & \alpha+2\beta & 1 & 0 & 0 & 0 & \ldots \\ 0 &
\beta(\alpha+\beta)
& \alpha+2\beta  & 1 & 0 &
0 & \ldots \\ 0 & 0 & \beta(\alpha+\beta)  & \alpha+2\beta  & 1 & 0 & \ldots \\ 0 & 0 & 0
& \beta(\alpha+\beta)  & \alpha+2\beta & 1 & \ldots \\0 & 0 & 0 & 0 & \beta(\alpha+\beta)  & \alpha+2\beta
&\ldots\\
\vdots &
\vdots & \vdots & \vdots & \vdots & \vdots &
\ddots\end{array}\right),\end{displaymath}
from which we deduce the three-term recurrence.
\end{proof}
\begin{corollary}
$$\left(\frac{1-\beta x}{1+\alpha x}, \frac{x(1-\beta x)}{1+\alpha x}\right)=
\left(1,\frac{x}{1-\beta x}\right)\cdot \left(\frac{1}{1+(\alpha+\beta)x},\frac{x}{(1+\beta x)(1+(\alpha+\beta )x}\right).$$
\end{corollary}
\begin{corollary} We have the relations $$P_n(x)=\sum_{k=0}^n \binom{n-1}{n-k}\beta^{n-k}\tilde{P}_k(x), \quad \quad \tilde{P}_n(x)=\sum_{k=0}^n \binom{n-1}{n-k}(-\beta)^{n-k}P_k(x).$$
\end{corollary}
\noindent We finish this section by noting that
$$\left(\frac{1-\beta x}{1+\alpha x}, \frac{x(1-\beta x)}{1+\alpha x}\right)\cdot \frac{1}{1-tx}=\frac{1-\beta x}{1+(\alpha-t)x+\beta t x^2},$$ and hence we have the generating function
$$\frac{1-\beta x}{1+(\alpha-t)x+\beta t x^2}=\sum_{n=0}^{\infty} P_n(t)x^n.$$
\section{Moments and T-fractions}
We identify the moments of $P_n(x)$ and $\tilde{P}_n(x)$ with the first column elements of the inverses of their coefficient arrays. The moments coincide, since they both have the generating function
 $$\frac{1-\alpha x- \sqrt{1-(\alpha+2\beta)x+\alpha^2 x^2}}{2 \beta x}.$$ In order to find a closed expression for these moments, we calculate the inverse of $\left(\frac{1-\beta x}{1+\alpha x},\frac{x(1-\beta x)}{1+\alpha x}\right)$. We use Lagrange inversion \cite{Merlini_When} for this. Given that $\left(\frac{1-\beta x}{1+\alpha x},\frac{x(1-\beta x)}{1+\alpha x}\right)$ is an element of the Bell subgroup of the Riordan group, its inverse will be of the form $\left(\frac{v(x)}{x}, v(x)\right)$, where  $$v(x)=\text{Rev}\left(\frac{x(1-\beta x)}{1+\alpha x}\right).$$ \noindent
Then we have
\begin{eqnarray*}
[x^n]\frac{v(x)}{x} v(x)^k &=& [x^{n+1}] v(x)^{k+1}\\
&=& [x^{n+1}] \left(\text{Rev}\left(\frac{x(1-\beta x)}{1+\alpha x}\right)\right)^{k+1}\\
&=&\frac{1}{n+1} [x^n] (k+1) x^k\left(\frac{1+\alpha x}{1-\beta x}\right)^{n+1}\\
&=& \frac{k+1}{n+1}[x^{n-k}] \left(\frac{1+\alpha x}{1-\beta x}\right)^{n+1}\\
&=& \frac{k+1}{n+1} [x^{n-k}] \sum_{j=0}^{n+1} \alpha^j x^j \sum_{i=0}^{\infty} \binom{-(n+1)}{i}(-1)^i \beta^i x^i\\
&=& \frac{k+1}{n+1} \sum_{j=0}^{n+1} \binom{n+1}{j}\binom{2n-k-j}{n-k-j}\alpha^j \beta^{n-k-j}.\end{eqnarray*}
\noindent

Thus the moments $\mu_n$ of the family of Laurent biorthogonal polynomials $P_n(x)$ are given by
$$\mu_n=\frac{1}{n+1} \sum_{j=0}^{n+1} \binom{n+1}{j}\binom{2n-j}{n-j}\alpha^j \beta^{n-j}.$$ A more compact form is
$$\mu_n=\sum_{k=0}^n \binom{n+k}{2k}C_k \alpha^{n-k} \beta^k,$$ where $C_n=\frac{1}{n+1}\binom{2n}{n}$ is the $n$-the Catalan number \seqnum{A000108}.
The theory of $T$-fractions and Laurent biorthogonal polynomials now tells us that the moments $\mu_n$ are generated by the following continued fraction \cite{Wall}.
$$g_{\alpha,\beta}(x)=
\cfrac{1}{1-\alpha x-
\cfrac{\beta x}{1-\alpha x-
\cfrac{\beta x}{1-\alpha x- \cdots}}}.$$
Now the sequence $\mu_n$ is also the moment sequence for the orthogonal polynomials $\tilde{P}_n(x)$. Hence we also have the following Stieltjes continued fraction for $g_{\alpha, \beta}(x)$.
$$g_{\alpha, \beta}(x)=
\cfrac{1}{1-(\alpha+\beta)x-
\cfrac{\beta(\alpha+\beta)x^2}{1-(\alpha+2\beta)x-
\cfrac{\beta(\alpha+\beta)x^2}{1-(\alpha+2\beta)x-\cdots}}}.$$ \noindent From this we deduce the following result concerning the Hankel transform of the moments $\mu_n$ \cite{Layman}.
\begin{proposition} The Hankel transform of the moments $\mu_n$ of $P_n(x)$ is given by
$$h_n=(\beta(\alpha+\beta))^{\binom{n+1}{2}}.$$
\end{proposition}
\begin{example} The large Schr\"oder numbers \seqnum{A006318}
$$S_n=\sum_{k=0}^n \binom{n+k}{2k}C_k$$ are the case $\alpha=\beta=1$. Thus the generating function for the large Schr\"oder numbers may be written
$$
\cfrac{1}{1- x-
\cfrac{x}{1- x-
\cfrac{ x}{1- x- \cdots}}}.$$
\noindent It is a classical result that the Hankel transform of the large Schr\"oder numbers is $2^{\binom{n+1}{2}}$.
\end{example}
\noindent
It is interesting to note that the row sums of the inverse matrix
$\left(\frac{1-\beta x}{1+\alpha x}, \frac{x(1-\beta x)}{1+\alpha x}\right)^{-1}$ have a generating function that can be expressed as the continued fraction
$$g_{\alpha, \beta}(x)=
\cfrac{1}{1-(\alpha+\beta+1)x-
\cfrac{\beta(\alpha+\beta)x^2}{1-(\alpha+2\beta)x-
\cfrac{\beta(\alpha+\beta)x^2}{1-(\alpha+2\beta)x-\cdots}}}.$$
Hence they too have a Hankel transform given by
$$h_n=(\beta(\alpha+\beta))^{\binom{n+1}{2}}.$$
The form of the continued fraction shows that the row sums are also moments for a family of orthogonal polynomials whose parameters can be read from the continued fraction.

Finally, the bi-variate generating function for the moment matrix $\left(\frac{1-\beta x}{1+\alpha x}, \frac{x(1-\beta x)}{1+\alpha x}\right)^{-1}$ is given by the generating function
$$
\cfrac{1}{1-(\alpha+\beta)x-xy-
\cfrac{\beta(\alpha+\beta)x^2}{1-(\alpha+2\beta)x-
\cfrac{\beta(\alpha+\beta)x^2}{1-(\alpha+2\beta)x-\cdots}}}.$$

\begin{example} The orthogonal polynomials $\tilde{P}_n(x;1,1)$ associated with the Riordan array $\left(\frac{1-\beta x}{1+\alpha x},\frac{x(1-\beta x)}{1+\alpha x}\right)$ have coefficient array given by
$$\left(\frac{1}{1+2x}, \frac{x}{1+3x+2x^2}\right).$$
\noindent The inverse of this matrix begins
\begin{displaymath}\left(\begin{array}{ccccccc} 1 & 0 & 0 & 0
&
0 & 0 & \cdots \\2 & 1 & 0 & 0 & 0 & 0 & \cdots \\ 6 & 5 & 1 &
0 & 0 & 0 &
\cdots \\ 22 & 23 & 8 & 1 & 0 & 0 & \cdots \\ 90 & 107 & 49 & 11 &
1 & 0 & \cdots \\394 & 509  & 276 & 84 & 14 & 1 &\cdots\\ \vdots
& \vdots &
\vdots & \vdots & \vdots & \vdots &
\ddots\end{array}\right),\end{displaymath} indicating that the large Schr\"oder numbers are the moments for this
family of orthogonal polynomials. This is borne out by the fact that the inverse is given by
$$\left(\frac{1-x-\sqrt{1-6x+x^2}}{2x}, \frac{1-3x-\sqrt{1-6x+x^2}}{4x}\right).$$
\noindent This is \seqnum{A133367}.
\end{example}

\section{Derivatives}
We consider the derivatives
$$\frac{d}{dx} P_n(x)=R_{n-1}(x), \quad n>0.$$
Since $P_n(x)$ is precisely of degree $n$, $R_n(x)$ is also of degree $n$, with $$R_0(x)=\frac{d}{dx}(x-(\alpha+\beta))=1.$$
We let $e_{n,k}$ be the $(n,k)$-th term of the coefficient array of $R_n(x)$. Thus we have
$$R_n(x)=\frac{d}{dx}P_{n+1}(x)=\sum_{k=0}^n e_{n,k}x^k.$$
\begin{proposition} The coefficient array whose $(n,k)$-th term is $\frac{1}{n+1} e_{n,k}$ is given by the Riordan array
$$\left(\left(\frac{1-\beta x}{1+\alpha x}\right)^2, \frac{x(1-\beta x)}{1+\alpha x}\right).$$
\end{proposition}
\noindent This results from a general result concerning Riordan arrays. We first note that we have
\begin{displaymath}\left(\begin{array}{ccccccc} 0 & 0 & 0 & 0
&
0 & 0 & \cdots \\1 & 0 & 0 & 0 & 0 & 0 & \cdots \\ 0 & 2 & 0 &
0 & 0 & 0 &
\cdots \\ 0 & 0 & 3 & 0 & 0 & 0 & \cdots \\ 0 & 0 & 0 & 4 &
0 & 0 & \cdots \\0 & 0  & 0 & 0 & 5 & 0 &\cdots\\ \vdots
& \vdots &
\vdots & \vdots & \vdots & \vdots &
\ddots\end{array}\right)\cdot
\left(\begin{array}{c}1\\x\\x^2\\x^3\\x^4\\x^5\\\vdots\end{array}\right)
=
\left(\begin{array}{c}0\\1\\2x\\3x^2\\4x^3\\5x^4\\\vdots\end{array}\right).\end{displaymath}
Thus we have
\begin{displaymath}\left(\begin{array}{ccccccc} d_{0,0} & 0 & 0 & 0
&
0 & 0 & \cdots \\d_{1,0} & d_{1,1} & 0 & 0 & 0 & 0 & \cdots \\ d_{2,0} & d_{2,1} & d_{2,2} &
0 & 0 & 0 &
\cdots \\ d_{3,0} & d_{3,1} & d_{3,2} & d_{3,3} & 0 & 0 & \cdots \\ d_{4,0} & d_{4,1} & d_{4,2} & d_{4,3} &
d_{4,4} & 0 & \cdots \\d_{5,0} & d_{5,1}  & d_{5,2} & d_{5,3} & d_{5,4} & d_{5,5} &\cdots\\ \vdots
& \vdots &
\vdots & \vdots & \vdots & \vdots &
\ddots\end{array}\right)\cdot
\left(\begin{array}{ccccccc} 0 & 0 & 0 & 0
&
0 & 0 & \cdots \\1 & 0 & 0 & 0 & 0 & 0 & \cdots \\ 0 & 2 & 0 &
0 & 0 & 0 &
\cdots \\ 0 & 0 & 3 & 0 & 0 & 0 & \cdots \\ 0 & 0 & 0 & 4 &
0 & 0 & \cdots \\0 & 0  & 0 & 0 & 5 & 0 &\cdots\\ \vdots
& \vdots &
\vdots & \vdots & \vdots & \vdots &
\ddots\end{array}\right)\cdot
\left(\begin{array}{c}1\\x\\x^2\\x^3\\x^4\\x^5\\\vdots\end{array}\right)
=
\left(\begin{array}{c}0\\R_0(x)\\R_1(x)\\R_2(x)\\R_3(x)\\R_4(x)\\\vdots\end{array}\right),\end{displaymath}
where in this case $\left(d_{n,k}\right)$ represents the Riordan array
$(g, f)=\left(\frac{1-\beta x}{1+\alpha x}, \frac{x(1-\beta x)}{1+\alpha x}\right)$.  We then have
\begin{displaymath}\left(\begin{array}{ccccccc} d_{1,1} & 0 & 0 & 0
&
0 & 0 & \cdots \\d_{2,1} & d_{2,2} & 0 & 0 & 0 & 0 & \cdots \\ d_{3,1} & d_{3,2} & d_{3,3} &
0 & 0 & 0 &
\cdots \\ d_{4,1} & d_{4,2} & d_{4,3} & d_{4,4} & 0 & 0 & \cdots \\ d_{5,1} & d_{5,2} & d_{5,3} & d_{5,4} &
d_{5,5} & 0 & \cdots \\d_{6,1} & d_{6,2}  & d_{6,3} & d_{6,4} & d_{6,5} & d_{6,6} &\cdots\\ \vdots
& \vdots &
\vdots & \vdots & \vdots & \vdots &
\ddots\end{array}\right)\cdot
\left(\begin{array}{ccccccc} 1 & 0 & 0 & 0
&
0 & 0 & \cdots \\0 & 2 & 0 & 0 & 0 & 0 & \cdots \\ 0 & 0 & 3 &
0 & 0 & 0 &
\cdots \\ 0 & 0 & 0 & 4 & 0 & 0 & \cdots \\ 0 & 0 & 0 & 0 &
5 & 0 & \cdots \\0 & 0  & 0 & 0 & 0 & 6 &\cdots\\ \vdots
& \vdots &
\vdots & \vdots & \vdots & \vdots &
\ddots\end{array}\right)\cdot
\left(\begin{array}{c}1\\x\\x^2\\x^3\\x^4\\x^5\\\vdots\end{array}\right)
=
\left(\begin{array}{c}R_0(x)\\R_1(x)\\R_2(x)\\R_3(x)\\R_4(x)\\R_5(x)\\\vdots\end{array}\right).\end{displaymath}
\noindent By the Riordan array structure of $(g, f)$, the leftmost matrix is given by $(gf, f)$. In other words,
we have
$$\left(\left(\frac{1-\beta x}{1+\alpha x}\right)^2, \frac{x(1-\beta x)}{1+\alpha x}\right) \cdot \text{Diag}(1,2,3,\ldots)=\left(\begin{array}{c}R_0(x)\\R_1(x)\\R_2(x)\\R_3(x)\\R_4(x)\\R_5(x)\\\vdots\end{array}\right).$$
\section{Generalized orthogonality}
Ismail and Masson \cite{I_M} have studied a notion of generalized orthogonality, associated with polynomials that satisfy recurrences of the type
$$P_n(x)=(x-c_n)P_{n-1}(x)-\lambda_n(x-a_n)P_{n-2}.$$
\noindent We specialize this to the constant coefficient case. Thus in this section we let $P_n(x)=P_n(x;\alpha, \beta, \gamma)$ be a family of polynomials that obeys the recurrence
$$P_n(x)=(x-\alpha)P_{n-1}(x)-\beta(x-\gamma)P_{n-2}(x),$$ with
$P_0(x)=1$ and $P_1(x)=x-\alpha-\beta$.

Following a similar development to that in the first section, we arrive at the following results.
\begin{proposition} The coefficient array of the polynomials $P_n(x)$ is given by the Riordan array
$$\left(\frac{1-\beta x}{1+\alpha x - \beta \gamma x^2}, \frac{x(1-\beta x)}{1+\alpha x - \beta \gamma x^2}\right).$$
\end{proposition}
\noindent Note that when $\gamma=0$, we retrieve our original case. When $\alpha+\beta \ne \gamma$, we have
\begin{proposition} The moments $\mu_n$ of the family of polynomials $P_n(x)$ are also the moments of the family of associated orthogonal polynomials $\tilde{P}_n(x)$ whose coefficient array is given by the Riordan array
$$\left(\frac{1+\beta x}{1+(\alpha+2\beta)x+\beta(\alpha+\beta-\gamma)x^2},\frac{x}{1+(\alpha+2\beta)x+\beta(\alpha+\beta-\gamma)x^2}\right).$$
\end{proposition}
\begin{corollary}
The Hankel transform of the moments $\mu_n$ is given by
$$h_n=(\beta(\alpha+\beta-\gamma))^{\binom{n+1}{2}}.$$
\end{corollary}
\noindent The relationship between the generalized orthogonal polynomials $P_n(x)$ of this section and their associated orthogonal polynomials $\tilde{P}_n(x)$ is given by the following proposition, which can be verified by applying the multiplication rule for Riordan arrays.
\begin{proposition} \footnotesize
$$ \left(\frac{1-\beta x}{1+\alpha x - \beta \gamma x^2}, \frac{x(1-\beta x)}{1+\alpha x - \beta \gamma x^2}\right)=
\tilde{B}\cdot \left(\frac{1+\beta x}{1+(\alpha+2\beta)x+\beta(\alpha+\beta-\gamma)x^2},\frac{x}{1+(\alpha+2\beta)x+\beta(\alpha+\beta-\gamma)x^2}\right),$$ \normalsize where
$$\tilde{B}=\left(1, \frac{x}{1-\beta x}\right).$$
\end{proposition}
\begin{example}
We consider the polynomials $P_n(x)=P_n(x;2,1,1)$ with coefficient array
$$\left(\frac{1-x}{1+2x-x^2}, \frac{x(1-x)}{1+2x-x^2}\right).$$ \noindent This begins
\begin{displaymath}\left(\begin{array}{ccccccc} 1 & 0 & 0 & 0
&0 & 0 & \cdots \\-3 & 1 & 0 & 0 & 0 & 0 & \cdots \\ 7 & -6 & 1 &
0 & 0 & 0 &
\cdots \\ -17 & 23 & -9 & 1 & 0 & 0 & \cdots \\ 41 & -76 & 48 & -12 &
1 & 0 & \cdots \\-99 & 233  & -204 & 82 & -15 & 1 &\cdots\\ \vdots
& \vdots &
\vdots & \vdots & \vdots & \vdots &
\ddots\end{array}\right).\end{displaymath}
$P_n(x)$ satisfies the recurrence
$$P_n(x)=(x-2)P_{n-1}-(x-1)P_{n-2}$$
with $P_0(x)=1$ and $P_1(x)=x-3$.
The moment array is given by $\left(\frac{1-x}{1+2x-x^2}, \frac{x(1-x)}{1+2x-x^2}\right)^{-1}$ which begins
\begin{displaymath}\left(\begin{array}{ccccccc} 1 & 0 & 0 & 0
&0 & 0 & \cdots \\3 & 1 & 0 & 0 & 0 & 0 & \cdots \\11 & 6 & 1 &
0 & 0 & 0 &
\cdots \\ 47 & 31 & 9 & 1 & 0 & 0 & \cdots \\ 223 & 160 & 60 & 12 &
1 & 0 & \cdots \\ 1135 & 849  & 366 & 98 & 15 & 1 &\cdots\\ \vdots
& \vdots &
\vdots & \vdots & \vdots & \vdots &
\ddots\end{array}\right).\end{displaymath}
Thus the moment sequence $\mu_n$ starts
$$1,3,11,47,223,1135,\ldots$$
\noindent This is the binomial transform \seqnum{A174347} of the large Schr\"oder numbers.
$$\mu_n=\sum_{k=0}^n \binom{n}{k} S_k.$$
\noindent This is a consequence of the fact that
\begin{equation}\label{Eq}\left(\frac{1-x}{1+2x-x^2}, \frac{x(1-x)}{1+2x-x^2}\right) \cdot B = \left(\frac{1-x}{1+x}, \frac{x(1-x)}{1+x}\right).\end{equation}
\noindent We note that $\left(\frac{1-x}{1+x}, \frac{x(1-x)}{1+x}\right)$ is the coefficient array of the Laurent biorthogonal  polynomials
$$Q_n(x)=(x-1)Q_{n-1}-xQ_{n-2}.$$
Equation (\ref{Eq}) is equivalent to
$$P_n(x+1)=Q_n(x).$$
\end{example}
\noindent We can also transform the coefficient array $ \left(\frac{1-\beta x}{1+\alpha x - \beta \gamma x^2}, \frac{x(1-\beta x)}{1+\alpha x - \beta \gamma x^2}\right)$ of $P_n(x;\alpha, \beta, \gamma)$ by multiplication on the right to obtain the coefficient array of an associated family of Laurent biorthogonal polynomials. This is the content of the next result, which is verifiable by straight-forward Riordan array multiplication.
\begin{proposition} We have
$$\left(\frac{1-\beta x}{1+\alpha x - \beta \gamma x^2}, \frac{x(1-\beta x)}{1+\alpha x - \beta \gamma x^2}\right) \cdot B_{\gamma} = \left(\frac{1-\beta x}{1+(\alpha-\gamma)x}, \frac{x(1-\beta x)}{1+(\alpha-\gamma)x}\right),$$ where
 $$B_{\gamma}=\left(\frac{1}{1-\gamma x},\frac{x}{1-\gamma x}\right).$$
\end{proposition}
\begin{corollary} Using an obvious notation, we have
$$P_n(x+\gamma; \alpha, \beta, \gamma)=P_n(x; \alpha-\gamma, \beta).$$
\end{corollary}
\begin{proof} This follows since
$$B_{\gamma}\cdot (1,x,x^2,\ldots)^T=(1,x+\gamma, (x+\gamma)^2, \ldots)^T.$$
\end{proof}

\section{Determinant representations}
It is possible to give determinant representations for the polynomial sequences \cite{Yang} in this note.

Taking the polynomials defined by the Riordan array $M=\left(\frac{1-\beta x}{1+\alpha x}, \frac{x(1-\beta x}{1+\alpha x}\right)$, we calculate $A(x)$ and $Z(x)$ for the inverse $\left(\frac{1-\beta x}{1+\alpha x}, \frac{x(1-\beta x}{1+\alpha x}\right)^{-1}$ to get
$$A(x)=\frac{1+\alpha x}{1-\beta x}, \quad \quad Z(x)=\frac{\alpha + \beta }{1-\beta x}.$$
\noindent Thus the production matrix of the inverse  $M^{-1}$  begins
\begin{displaymath}\left(\begin{array}{ccccccc} \alpha+\beta & 1 &
0 & 0 & 0 & 0 & \ldots \\
\beta(\alpha+\beta)  & \alpha+\beta & 1 & 0 & 0 & 0 & \ldots \\
\beta^2(\alpha+\beta)  &\beta(\alpha+\beta) & \alpha+\beta  & 1 & 0 & 0 & \ldots \\
\beta^3(\alpha+\beta) & \beta^2(\alpha+\beta) & \beta(\alpha+\beta)  & \alpha+\beta  & 1 & 0 & \ldots \\
\beta^4(\alpha+\beta) & \beta^3(\alpha+\beta) & \beta^2(\alpha+\beta)
& \beta(\alpha+\beta)  & \alpha+\beta & 1 & \ldots \\

\beta^5(\alpha+\beta) & \beta^4(\alpha+\beta) & \beta^3(\alpha+\beta) & \beta^2(\alpha+\beta) & \beta(\alpha+\beta)  & \alpha+\beta
&\ldots\\
\vdots &
\vdots & \vdots & \vdots & \vdots & \vdots &
\ddots\end{array}\right),\end{displaymath}
\noindent
Letting $\mathbf{P}_n$ denote the principal submatrix of order $n$ of this production matrix, we have
$$P_n(x)=P_n(x;\alpha, \beta)=\det(x I_n-\mathbf{P}_n),$$
where $I_n$ is the identity matrix of order $n$.

Similar remarks hold for the other polynomials defined in this article.

\bigskip \hrule \bigskip

\noindent 2010 {\it Mathematics Subject Classification}: Primary 11C20; Secondary 11B83, 15B36, 33C45.
\\
\noindent \emph{Keywords:
Riordan array, Laurent biorthogonal polynomials, orthogonal polynomials, generalized orthogonal polynomial, A sequence, Z sequence, production matrix}.

 \bigskip \hrule \bigskip \noindent Concerns sequences

\seqnum{A000045},
\seqnum{A006318},
\seqnum{A080246},
\seqnum{A080247},
\seqnum{A133367},
\seqnum{A174347}.

\end{document}